\documentclass[11pt,leqno]{amsart}

\usepackage{amsthm}
\usepackage{amsmath}
\usepackage{amssymb}
\usepackage[all,cmtip]{xy}
\usepackage[T1]{fontenc}
\usepackage{mathrsfs}
\usepackage{enumerate}

\usepackage{color}

\theoremstyle{plain}
\newtheorem{thm}{Theorem}

\newtheorem{prop}{Proposition}
\newtheorem{cor}{Corollary}
\theoremstyle{definition}
\newtheorem{defn}{Definition}

\newtheorem{exmp}{Example}

\newtheorem{rem}{Remark}

\DeclareMathOperator{\ob}{ob}
\DeclareMathOperator{\mor}{mor}

\title{Elementary Magma Gradings on Rings}
\date{}

\author{Patrik Lundström}

\address{Patrik Lundstr\"{o}m,
University West,
Department of Engineering Science,
Trollh\"{a}ttan,
Sweden}
\email{Patrik.Lundstrom@hv.se}

\begin{document}

\maketitle


\begin{abstract}
Suppose that $G$ and $H$ are magmas and that $R$ is a strongly $G$-graded ring.
We show that there is a bijection between the set of elementary
(nonzero) $H$-gradings of $R$ and the set of (zero) magma homomorphisms from $G$ to $H$.
Thereby we generalize a result by D\u{a}sc\u{a}lescu,
N\u{a}st\u{a}sescu and Rios Montes from group gradings of matrix rings
to strongly magma graded rings. We also show that
there is an isomorphism between the preordered set of elementary (nonzero) $H$-filters on $R$
and the preordered set of (zero) submagmas of $G \times H$.
These results are applied to category graded rings
and, in particular, to the case when $G$ and $H$ are groupoids.
In the latter case, we use this bijection to determine the cardinality
of the set of elementary $H$-gradings on $R$.
\end{abstract}

\section{Introduction}

A classical problem in ring theory is to find all group gradings
on matrix rings (see e.g. \cite{bah01}, \cite{bah02}, \cite{bob1}, \cite{bob2},
\cite{cae02}, \cite{das99}, \cite{kel02}, \cite{nas82}, \cite{nas04}, \cite{zai01}).
More precisely, if $H$ is a group,
$n$ is a positive integer, $K$ is a field and $M_n(K)$ denotes the ring of $n \times n$
matrices with entries in $K$, then an $H$-\emph{grading} on $M_n(K)$ is a
collection of $K$-subspaces $(W_h)_{h \in H}$ of
$M_n(K)$ satisfying $M_n(K) = \oplus_{h \in H} W_h$ and
$W_h W_{h'} \subseteq W_{hh'}$ for all $h,h' \in H$.
In \cite{das99} D\u{a}sc\u{a}lescu, N\u{a}st\u{a}sescu and Rios Montes
determine all $H$-gradings on $M_n(K)$
in the case when $H$ is a finite cyclic group with $q$ elements
and $K$ contains a primitive $q$th root of 1.
In loc. cit. D\u{a}sc\u{a}lescu et. al. also determine all
cyclic gradings of order two on $M_2(K)$ for any field $K$.
In \cite{bah01} Bahturin and Sehgal describe all $H$-gradings
on $M_n(K)$ when $H$ is abelian and $K$ is an algebraically closed field.
The general case is still unsettled.
However, if we restrict ourselves to the problem of
finding all \emph{elementary} $H$-gradings on $M_n(K)$,
then there is a complete answer to this problem
(see Theorem \ref{maintheorem}).
Namely, following the notation introduced in \cite{bah01},
an $H$-grading on $M_n(K)$ is called elementary if
each $W_h$, for $h \in H$, considered
as a left vector space over $K$, has a
basis consisting of matrices of the type $e_{i,j}$
with 1 at the $(i,j)$ position and 0 elsewhere.
Note that the concept of elementary gradings of matrix algebras
have appeared in another setting in the work of
E. L. Green and E. N. Marcos (see \cite{gre83} and \cite{gre94})
where they view $M_n(K)$ as a quotient of the path
algebra of the complete graph on $n$ points,
the elementary gradings arise from weight functions on this graph.

\begin{thm}[D\u{a}sc\u{a}lescu, N\u{a}st\u{a}sescu and Rios Montes \cite{das99}]\label{maintheorem}
If $H$ is a finite group with $q$ elements, then there are $q^{n-1}$ elementary $H$-gradings on
the matrix algebra $M_n(K)$.
\end{thm}

The main purpose of this article is to generalize Theorem \ref{maintheorem}
from group gradings on matrix algebras to
groupoid gradings on groupoid algebras (see Theorem \ref{groupoids}).
The impetus for this generalization
is the observation that $M_n(K)$ is a groupoid algebra over $K$.
Namely, suppose that $\Gamma$ is a category.
The objects and morphisms in $\Gamma$ are denoted by
$\ob(\Gamma)$ and $\mor(\Gamma)$ respectively.
The domain and codomain of $s \in \mor(\Gamma)$
are denoted $d(s)$ and $c(s)$ respectively.
If $e \in \ob(\Gamma)$, then we let $\Gamma_e$ denote
the monoid of all morphisms from $e$ to $e$.
In particular, we let ${\rm id}_e$ denote the identity motphism $e \rightarrow e$.
If $e,e' \in \ob(\Gamma)$, then we let $\hom(e,e')$ denote the 
set of morphisms $e \rightarrow e'$.
Recall that $\Gamma$ is called a \emph{groupoid} if all of its
morphisms are isomorphisms. Note that if $\Gamma$
is a groupoid, then $\Gamma_e$ is a group for all $e \in \ob(\Gamma)$.
A category $\Gamma$ is called \emph{thin} (\emph{connected}) if for all
$x,y \in \ob(\Gamma)$, there is at most (at least) one morphism
from $x$ to $y$.
A category is called \emph{finite} if the set of morphisms of the category is finite.
Recall that the category algebra of $\Gamma$ over $K$, denoted $K[\Gamma]$,
is the collection of formal sums
$\sum_{s \in \mor(\Gamma)} a_s s$, where $a_s \in K$, for $s \in \mor(\Gamma)$,
are chosen so that all but finitely many of them are nonzero.
The addition and multiplication on $K[\Gamma]$ is defined by
$\sum_{s \in \mor(\Gamma)} a_s s + \sum_{s \in \mor(\Gamma)} b_s s = \sum_{s \in \mor(\Gamma)} (a_s + b_s)s$
and
$\sum_{s \in \mor(\Gamma)} a_s s \sum_{t \in \mor(\Gamma)} b_t t = \sum_{r \in \mor(\Gamma)}
\sum_{d(s)=c(t), \ st=r} a_s b_t r$
respectively for all $\sum_{s \in \mor(\Gamma)} a_s s \in K[\Gamma]$ and
all $\sum_{t \in \mor(\Gamma)} b_t t \in K[\Gamma]$.
If $\Gamma$ is a groupoid, a monoid or a group, then
$K[\Gamma]$ is called a groupoid algebra, a monoid algebra or a group algebra
respectively.
Note that if $\Gamma$ is a connected groupoid,
then the cardinality of $\hom(\Gamma_e,\Gamma_{e'})$
is the same for all choices of $e,e' \in \ob(\Gamma)$.
If $\Lambda$ also is a connected groupoid,
then the cardinality of $\hom(\Gamma_e,\Lambda_f)$
is the same for all choices of $e \in \ob(\Gamma)$
and $f \in \ob(\Lambda)$.
For the details concerning these claims, see the discussion in the end of
Section \ref{categoryfilteredrings} where we show the following result.

\begin{thm}\label{groupoids}
Let $\Gamma$ and $\Lambda$ be finite connected groupoids
and take $e \in \ob(\Gamma)$ and $f \in \ob(\Lambda)$.
If we put $m = |\ob(\Gamma)|$, $n = |\ob(\Lambda)|$,
$p = |\hom(\Gamma_e,\Lambda_f)|$ and $q = |\hom(\Gamma_e,\Gamma_e)|$,
then there are $(p q^{m-1})^{n^m}$ elementary $\Lambda$-gradings
on the groupoid algebra $K[\Gamma]$.
\end{thm}

Theorem \ref{groupoids} generalizes Theorem \ref{maintheorem}.
Indeed, if we let $\Gamma$ be the unique thin connected groupoid with $n$ objects,
that is, if we let $\Gamma$ have the first $n$ positive integers as objects
and the matrices $e_{i,j}$ as morphisms, where $d(e_{i,j})=j$ and $c(e_{i,j})=i$,
then $M_n(K)$ equals the groupoid algebra $K[\Gamma]$.
If we now let $\Lambda$ be a group,
that is a groupoid with one object,
then, by Theorem \ref{groupoids}, it follows that that there
are $(p q^{m-1})^{n^m} = (1 \cdot q^{n-1})^{1^m} = q^{n-1}$
elementary $\Lambda$-gradings on $K[\Gamma]$.
For the relevant definitions concerning elementary category gradings on rings,
see Section \ref{categoryfilteredrings}.
Note that one may, using Theorem \ref{groupoids},
count the cardinality of the set of
elementary $\Lambda$-gradings of $K[\Gamma]$
for any (not necessarily connected) finite groupoids $\Gamma$ and $\Lambda$
(see Remark \ref{lastremark} in Section \ref{categoryfilteredrings}).

Theorem \ref{maintheorem} is in \cite{das99} shown
by displaying an explicit bijection between the set of elementary $H$-gradings
on $M_n(K)$ and the set $H^{n-1}$.
This bijection is defined in three steps. Namely, suppose that we are given $h_i \in H$,
for $i = 1,2,\ldots,n-1$.
Step 1: If $1 \leq i \leq n-1$, then map $e_{i,i+1}$ to $h_i$.
Step 2: If $1 \leq i < j \leq n$, then map $e_{i,j}$ to $h_i h_{i+1} \cdots h_{j-1}$.
Step 3: If $1 \leq j \leq i \leq n$, then map $e_{i,j}$ to
$h_{i-1}^{-1} h_{i-2}^{-1} \cdots h_j^{-1}$.
In Section \ref{categoryfilteredrings}, we generalize this bijection
to category graded and filtered rings (see Theorem \ref{categorymainthm}
and Theorem \ref{relcategorymainthm}).
In fact, we even show that the category graded bijection follows from the 
following general result that holds for certain magma graded rings.

\begin{thm}\label{genmaintheorem}
If $G$ and $H$ are (zero) magmas and $R$ is a ring equipped with a nonzero strong $G$-grading, then
there is a bijection between the set of elementary
(nonzero) $H$-gradings on $R$ and the set of (zero) magma
homomorphisms from $G$ to $H$.
\end{thm}

For a proof of Theorem \ref{genmaintheorem} and the
relevant definitions concerning
(zero) magmas and magma gradings on rings, see Section \ref{magmafilteredrings}.

In this article, we also introduce and analyse the problem
of finding all elementary magma \emph{filters} on a given ring.
For the definition of this concept, see Section \ref{magmafilteredrings}.
Note that magma filters on rings have been extensively
studied in the case when the magma equals the nonnegative integers
equipped with addition as operation.
Namely, in that case a filter on a ring is a collection $(W_n)_{n \geq 0}$
of additive subgroups of the ring satisfying $W_n W_{n'} \subseteq W_{n+n'}$
for all nonnegative integers $n$ and $n'$.
For more details concerning this, see any standard book
on commutative ring theory, e.g. Chapter III in \cite{bourbaki}.
We remark that in the literature there do not seem to exist
any results concerning the problem of finding
all magma filters on rings - not even in the case when the magma equals the
nonnegative integers.
In Section \ref{magmafilteredrings}, we show the following result.

\begin{thm}\label{relgenmaintheorem}
If $G$ and $H$ are (zero) magmas and $R$ is a ring equipped with a nonzero strong $G$-grading, then
there is an isomorphism between the preordered set of elementary (nonzero) $H$-filters on $R$
and the preordered set of (zero) submagmas of $G \times H$.
\end{thm}

In the end of Section \ref{magmafilteredrings}, we
apply Theorem \ref{genmaintheorem} and Theorem \ref{relgenmaintheorem}
on magma rings over $K$ (see Corollary \ref{cor1} and Corollary \ref{cor2}).
In particular, we use these bijections to determine the number
of magma gradings and magma filters on rings in some concrete cases
(see Examples \ref{groupexample}-\ref{example6}).
In Section \ref{categoryfilteredrings}, we apply
Theorem \ref{genmaintheorem} and Theorem \ref{relgenmaintheorem}
on category graded and category filtered rings (see Theorem \ref{categorymainthm}
and Theorem \ref{relcategorymainthm}) and exemplify this
in some concrete cases (see Examples \ref{example7}-\ref{lastexample}). 
In the end of this section, we show Theorem \ref{groupoids}.

\section{Magma Filtered Rings}\label{magmafilteredrings}

In this section, we recall some fairly well known notions from
the theory of magmas (see Definition \ref{definition-magma}).
For more details concerning this, see e.g.
the book \cite{kel02} by Kelarev
(where however the notion of magma, found e.g. in \cite{bourbaki},
is called groupoid).
We also introduce the concept of (elementary) magma filtered rings
(see Definition \ref{definition-ring}).
Then we show Theorem \ref{genmaintheorem}
and Theorem \ref{relgenmaintheorem} through
a series of results (see Propositions \ref{fGH}-\ref{FM})
some of which hold in a more general context.
We also introduce the notions of zero magma and
zero magma homomorphism (see Definition \ref{definition-zeromagma}).
In the end of this section, we apply Theorem \ref{genmaintheorem} and
Theorem \ref{relgenmaintheorem} to several
different cases of magma algebras (see Corollary \ref{cor1},
Corollary \ref{cor2} and Examples \ref{groupexample}-\ref{example6}).

\begin{defn}\label{definition-magma}
For the rest of the section, let $G$ be a magma.
By this we mean that $G$ is a set equipped
with a binary operation $G \times G \ni (g,g') \mapsto gg' \in G$.
We consider the empty set to be a magma.
By a submagma of $G$ we mean a subset of $G$ which is closed
under the binary operation on $G$.
Let $H$ be another magma.
The product set $G \times H$ has a natural
structure of a magma induced by the
binary operations on $G$ and $H$.
If $f$ is a subset of $G \times H$ and $h \in H$,
then we let $f^{-1}(h)$ denote the collection of $g \in G$
such that $(g,h) \in f$.
We let $S(G \times H)$ denote the partially ordered set of
submagmas of $G \times H$ ordered by inclusion.
We say that a function $f : G \rightarrow H$ is a
homomorphism of magmas if $f(gg') = f(g)f(g')$
for all $g,g' \in G$. We let $\hom(G,H)$
denote the set of magma homomorphisms from $G$ to $H$.
\end{defn}

\begin{defn}\label{definition-ring}
For the rest of the section, let $R$ be an associative ring
equipped with a $G$-\emph{filter} $(V_g)_{g \in G}$.
By this we mean that $(V_g)_{g \in G}$
is a collection of additive subgroups of $R$
satisfying $V_g V_{g'} \subseteq V_{gg'}$ for all $g,g' \in G$.
We say that such a $G$-filter is \emph{nonzero} (\emph{strong}) if
$V_g \neq \{ 0 \}$ for all $g \in G$ ($V_g V_{g'} = V_{gg'}$ for all $g,g' \in G$).
Furthermore, we say that an $H$-filter
$(W_h)_{h \in H}$ on $R$ is \emph{elementary}
(with respect to $(V_g)_{g \in G}$) if
$W_h = \sum_{V_g \subseteq W_h} V_g$
for all $h \in H$.
We let $G(R)$ denote the partially ordered set of $G$-filters
on $R$, the ordering defined by saying that
$(V_g)_{g \in G} \leq (W_g)_{g \in G}$
if $V_g \subseteq W_g$ for all $g \in G$.
We say that the $G$-filter $(V_g)_{g \in G}$ on $R$ is
a $G$-\emph{grading} if $R = \oplus_{g \in G} V_g$.
A $G$-grading on $R$ is called nonzero (strong)
if it is nonzero (strong) as a $G$-filter.
\end{defn}

\begin{prop}\label{fGH}
The function $F : S(G \times H) \rightarrow H(R)$
defined by $F(f)_h = \sum_{g \in f^{-1}(h)} V_g$, for $h \in H$
and $f \in S(G \times H)$,
is a homomorphism of partially ordered sets.
If $(V_g)_{g \in G}$ is a grading on $R$ and
$f : G \rightarrow H$ is a homomorphism of magmas,
then $F(f)$ is an $H$-grading on $R$.
\end{prop}

\begin{proof}
First we show that $F$ is well defined.
Take $h,h' \in H$ and a submagma $f$ of $G \times H$. Then
$$F(f)_h F(f)_{h'} = \sum_{ \tiny{\begin{array}{c}
                                g \in f^{-1}(h) \\
                                g' \in f^{-1}(h')
                              \end{array}} }
V_g V_{g'} \subseteq \sum_{ \tiny{\begin{array}{c}
                                g \in f^{-1}(h) \\
                                g' \in f^{-1}(h')
                              \end{array}} } V_{gg'} \subseteq$$
$$\subseteq \sum_{g'' \in f^{-1}(hh')} V_{g''} = F(f)_{hh'}.$$
Therefore $F(f)$ is an $H$-filter on $R$.
Now we show that $F$ respects inclusion.
Suppose that $f'$ is another submagma of $G \times H$
such that $f \subseteq f'$. Then
$$F(f)_h = \sum_{g \in f^{-1}(h)} V_g \subseteq \sum_{g \in f'^{-1}(h)} V_g
= F(f')_h.$$
Therefore $F(f) \leq F(f')$. If $f : G \rightarrow H$ is a
homomorphism of magmas and $(V_g)_{g \in G}$
is a $G$-grading on $R$, then the sets
$f^{-1}(h)$, for $h \in H$, are pairwise disjoint
and cover $G$.
Therefore, we get that
$R = \oplus_{g \in G} V_g = \oplus_{h \in H}
\left( \oplus_{g \in f^{-1}(h)} V_g \right) =
\oplus_{h \in H} F(f)_h.$
\end{proof}

\begin{prop}\label{fGHH}
Suppose that the $G$-filter $(V_g)_{g \in G}$ on $R$ is strong.
Then the function $M : H(R) \rightarrow S(G \times H)$
defined by saying that $(g,h) \in M( (W_h)_{h \in H} )$,
for an $H$-filter $(W_h)_{h \in H}$ on $R$, precisely
when $V_g \subseteq W_h$, is a homomorphism of
partially ordered sets.
If $(V_g)_{g \in G}$ is nonzero, $(W_h)_{h \in H}$
is a grading on $R$ and there to each $g \in G$
is $h \in H$ with $V_g \subseteq W_h$,
then $M( (W_h)_{h \in H} )$ is a homomorphism of magmas.
\end{prop}

\begin{proof}
First we show that $M$ is well defined.
Suppose that $(W_h)_{h \in H}$ is an $H$-filter on $R$
and that $(g,h)$ and $(g',h')$ belong to $M( (W_h)_{h \in H} )$.
Then $V_g \subseteq W_h$ and $V_{g'} \subseteq W_{h'}$.
Hence, since the $G$-filter $(V_g)_{g \in G}$ is strong, we get that
$V_{gg'} = V_g V_{g'} \subseteq W_h W_{h'} \subseteq W_{hh'}$.
Therefore $(gg',hh') \in M( (W_h)_{h \in H} )$.

Now we show that $M$ respects the partial orders.
Suppose that $(W'_h)_{h \in H}$ is another $H$-filter on $R$
such that $(W_h)_{h \in H} \leq (W'_h)_{h \in H}$.
Take $(g,h) \in M( (W_h)_{h \in H} )$.
Then $V_g \subseteq W_h \subseteq W'_h$ and hence
$(g,h) \in M( (W'_h)_{h \in H} )$.
Therefore we get that
$M( (W_h)_{h \in H} ) \subseteq M( (W'_h)_{h \in H} )$.

Now suppose that $(V_g)_{g \in G}$ is nonzero,
$(W_h)_{h \in H}$ is an $H$-grading on $R$,
and there to each $g \in G$ is $h \in H$ with $V_g \subseteq W_h$.
We show is that $M( (W_h)_{h \in H} )$ is a function.
Seeking a contradiction, suppose that there are $h,h' \in H$
with $h \neq h'$ and $V_g \subseteq W_h$ and $V_g \subseteq W_{h'}$.
Then, since $(W_h)_{h \in H}$ is an $H$-grading on $R$,
we get that $\{ 0 \} \subsetneq V_g \subseteq W_h \cap W_{h'} = \{ 0 \}$
which is a contradiction.
\end{proof}

\begin{prop}\label{MF}
If $f \in S(G \times H)$, then $M(F(f)) \supseteq f$
with equality if $(V_g)_{g \in G}$ is a nonzero
strong $G$-grading on $R$.
\end{prop}

\begin{proof}
Take $f \in S(G \times H)$. Then
$M(F(f)) = \{ (g,h) \in G \times H \mid
V_g \subseteq \sum_{g' \in f^{-1}(h)} V_{g'} \}
\supseteq f$. Now suppose that $(V_g)_{g \in G}$ is a nonzero
strong $G$-grading on $R$. Take $(g,h) \in G \times H$
such that $V_g \subseteq \sum_{g' \in f^{-1}(h)} V_{g'}$.
Then $V_g = V_{g'}$ for some $g' \in f^{-1}(h)$.
But this implies that $g = g' \in f^{-1}(h)$
and hence $(g,h) \in f$.
Therefore $M(F(f)) \subseteq f$.
\end{proof}

\begin{prop}\label{FM}
If $(V_g)_{g \in G}$ is a strong $G$-filter on $R$
and $(W_h)_{h \in H}$ is an elementary $H$-filter on $R$,
then $F( M( (W_h)_{h \in H} ) ) = (W_h)_{h \in H}$.
\end{prop}

\begin{proof}
Take $h' \in H$. Then
$F( M( (W_h)_{h \in H} ) )_{h'} = \sum_{g \in M((W_h)_{h \in H})^{-1}(h')} V_g
= \sum_{V_g \subseteq W_{h'}} V_g = W_{h'}$.
\end{proof}

\begin{defn}\label{definition-zeromagma}
Let $G$ be a {\it zero magma}.
By this we mean that $G$ is equipped with a {\it zero element},
that is an element $0$ satisfying $0g = g0 = 0$ for all $g \in G$.
Let $H$ be another zero magma.
A subset $f$ of $G \times H$ is called a {\it zero submagma} if
$f^{-1}(0) = 0$ and for all $(g,h)$ and $(g',h')$ in $f$ with $gg' \neq 0$,
the element $(gg',hh')$ belongs to $f$.
Let  $S_0(G \times H)$ denote the partially ordered set of
zero submagmas of $G \times H$ ordered by inclusion.
A function $f : G \rightarrow H$ is called a {\it zero magma homomorphism}
if the graph of $f$ considered as a subset of $G \times H$
belongs to $S_0(G \times H)$. The collection of zero magma homomorphisms
$G \rightarrow H$ is denoted $\hom_0(G,H)$.
Note that a function $f : G \rightarrow H$
with the property that $f^{-1}(0)=0$ is a zero
magma homomorphism precisely when $f(gg') = f(g)f(g')$
for all $g,g' \in G$ with $gg' \neq 0$.
We let $G_0(R)$ denote the partially ordered set of $G$-filters
$(W_g)_{g \in G}$ on $R$ satisfying $W_0 = V_0$.
We say that such a $G$-filter is {\it nonzero} if $W_g \neq \{ 0 \}$
for all nonzero $g \in G$. A grading $(W_g)_{g \in G}$ on $R$
is called nonzero if it is nonzero as a $G$-filter on $R$.
\end{defn}

\begin{rem}
Not all zero magma homomorphisms
are magma homomorphisms.
In fact, let $G = \{ a,b,0 \}$ and $H = \{ c,0 \}$
be two magmas equipped with rules of composition
defined by
$a a = a$, $bb = b$, $ab = ba = a0 = 0a = b0 = 0b = 00 = 0$
and
$cc = c$, $c0 = 0c = 00 = 0$ respectively.
Define a function $f : G \rightarrow H$ by
$f(a) = f(b) = c$ and $f(0)=0$.
Since $f(aa)=f(a)=c=cc=f(a)f(a)$
and $f(bb)=f(b)=c=cc=f(b)f(b)$
it follows that $f$ is a zero magma homomorphism.
However, since $f(ab) = f(0) = 0 \neq c = cc = f(a)f(b),$
$f$ is not a magma homomorphism.
\end{rem}

\subsection*{Proof of Theorem \ref{genmaintheorem} and
Theorem \ref{relgenmaintheorem}.}
By Proposition \ref{MF} and Proposition \ref{FM}
we get that $MF = {\rm id}_{S(G \times H)}$
and $FM = {\rm id}_{H(R)}$
respectively. Therefore the ''nonzero'' version of Theorem \ref{relgenmaintheorem} follows.
The ''nonzero'' version of Theorem \ref{genmaintheorem} follows from the same propositions
by restriction of the maps $F$ and $M$ to respectively $\hom(G,H)$
and the set of elementary $H$-gradings on $R$.
The ''zero'' versions of Theorem \ref{relgenmaintheorem} and Theorem \ref{genmaintheorem}
follow from the easily checked fact
that $M_0 F_0 = {\rm id}_{S_0(G \times H)}$
and $F_0 M_0 = {\rm id}_{H_0(R)}$ where $F_0$ and $M_0$
denote the restrictions of $F$ and $M$ to $S_0(G \times H)$
and $H_0(R)$ respectively. \hfill $\square$

\begin{defn}
Let $G$ be a magma and $K$ a field.
Recall that the {\it magma algebra} $K[G]$ of $G$ over $K$
is the collection of formal sums
$\sum_{g \in G} a_g g$, where $a_g \in K$, for $g \in G$,
are chosen so that all but finitely many of them are nonzero.
The addition and multiplication on $K[G]$ is defined by
$
\sum_{g \in G} a_s  s
+
\sum_{t \in G} b_t t
= \sum_{g \in G} (a_g + b_g)g
$
and
$
\sum_{s \in G} a_s s
\sum_{t \in G} b_t t
= \sum_{g \in G}
\sum_{st=g} a_s b_t g$
respectively, for all $\sum_{s \in G} a_s s \in K[G]$ and
all $\sum_{t \in G} b_t t \in K[G]$.
For the rest of the article, we fix the strong $G$-grading
$(V_g)_{g \in G}$ on $K[G]$ defined by putting
$V_g = Kg$ for all $g \in G$.
\end{defn}

\begin{cor}\label{cor1}
If $G$ and $H$ are (zero) magmas,
then there is a bijection between the set of
elementary (zero) $H$-gradings on $K[G]$ and the
set of (zero) magma homomorphisms from $G$ to $H$.
\end{cor}

\begin{proof}
This follows immediately from Theorem \ref{genmaintheorem}.
\end{proof}

\begin{cor}\label{cor2}
If $G$ and $H$ are (zero) magmas,
then there is an isomorphism between the preordered set of
elementary (zero) $H$-filters on $K[G]$ and the preordered
set of (zero) submagmas of $G \times H$.
\end{cor}

\begin{proof}
This follows immediately from Theorem \ref{relgenmaintheorem}.
\end{proof}

\begin{exmp}\label{groupexample}
Suppose that $G$ and $H$ are finite groups.
By Corollary \ref{cor1}, the cardinality of the set of
elementary $H$-gradings on $K[G]$ equals $|{\rm hom}(G,H)|$.
The latter seems hard to compute for general groups
(see e.g. \cite{asa93} and \cite{yos93} for some results
concerning congruences involving this number).
However, for abelian groups $G$ and $H$ it is fairly easy.
In fact, if $G$ and $H$ are cyclic, then $|{\rm hom}(G,H)|$
equals the greatest common divisor of $|G|$ and $|H|$.
For the case of general abelian $G$ and $H$,
suppose that $p_i$ denotes the $i$th prime number and
$G = \oplus_{i \geq 1} \oplus_{j \geq 1}
\left( {\Bbb Z}_{p_i^j} \right)^{a_{ij}}$ and
$H = \oplus_{k \geq 1} \oplus_{l \geq 1}
\left( {\Bbb Z}_{p_k^l} \right)^{a_{kl}}$,
where $a_{ij}=a_{kl}=0$ for all but finitely many
$i$, $j$, $k$ and $l$.
Then $|{\rm hom}(G,H)| = \prod_{i \geq 1} \prod_{j,k \geq 1}
\left( p_i^{{\rm min}(j,k)} \right)^{a_{ij} a_{ik}}$.
\end{exmp}

\begin{rem}
From the discussion in Example \ref{groupexample} it follows that
$|{\rm hom}(G,H)| = |{\rm hom}(H,G)|$
for all finite abelian groups $G$ and $H$. It is not
clear to the author at present whether this equality
is true for all groups $G$ and $H$.
\end{rem}

\begin{exmp}\label{groupexample2}
Suppose that $G$ and $H$ are finite groups.
It is easy to see that a submagma of a finite group
is a subgroup. Therefore, by Corollary \ref{cor2},
the number of elementary $H$-filters on $K[G]$ equals
the number of subgroups of $G \times H$.
The latter seems hard to determine
(see \cite{suz67} for a discussion on the structure
of the lattice of subgroups of $G \times H$).
Therefore we here only discuss two particular cases of interest.

If $|G|$ and $|H|$ are relatively prime, then
it is easy to see that the subgroup lattice
of $G \times H$ equals the product of the subgroup lattices
of $G$ and $H$. In particular it follows that the number
of subgroups of $G \times H$ equals the product of the
number of subgroups of $G$ and the number of subgroups of $H$.

If $G$ and $H$ both are direct products of copies of
${\Bbb Z}_p$, for some prime number $p$, then
$G \times H$ is a vector space over ${\Bbb Z}_p$.
By a simple counting argument,
the number of ${\Bbb Z}_p$-subspaces of ${\Bbb Z}_p^n$
is $\sum_{k=1}^n A_k/B_k$ where
$A_k = (p^n-1)(p^n-p)...(p^n-p^{k-1})$
and $B_k = (p^k-1)(p^k-p)...(p^k-p^{k-1}).$
\end{exmp}

\begin{exmp}\label{example3}
There are 10 nonisomorphic magmas of order two.
In fact, let the two elements of such a magma be denoted $a$ and $b$
and let the notation $x_1 x_2 x_3 x_4$ mean that the magma operation is defined by
the relations $a^2=x_1$, $ab=x_2$, $ba=x_3$ and $b^2=x_4$.
Furthermore, let $t$ denote the nonidentity bijection from
$\{ a,b \}$ to $\{ a,b \}$.
Then it is easy to see that
$t : x_1 x_2 x_3 x_4 \rightarrow y_1 y_2 y_3 y_4$
precisely when $y_1 = t(x_4)$, $y_2 = t(x_3)$,
$y_3 = t(x_2)$ and $y_4 = t(x_1)$.
Using this it is a straightforward task to verify that
a complete set of representatives for the different
isomorphism classes of magmas of order two is
$aaaa$, $baaa$, $abaa$, $aaba$, $aaab$,
$aabb$, $bbaa$, $abab$, $baba$ and $abba$.
Now we determine $\hom(G,H)$ for all 100 choices
of two element magmas $G$ and $H$.
In this description, we let $x_1 x_2 x_3 x_4 y_1 y_2 y_3 y_4$ denote the
set $\hom(x_1 x_2 x_3 x_4 , y_1 y_2 y_3 y_4)$.
We let $1$ denote the identity function $\{ a , b \} \rightarrow \{ a , b \}$
and  we let $a$ (or $b$) denote the constant function $\{ a , b \} \rightarrow \{ a , b \}$
that maps both $a$ and $b$ to $a$ (or $b$).
{\small $$aaaa aaaa = \{ 1,a \} \quad aaaa baaa = \emptyset \quad
aaaa abaa = \{ a \} \quad aaaa aaba = \{ a \}$$
$$aaaa aaab = \{ a \} \quad aaaa aabb = \{ a,b \} \quad aaaa bbaa = \emptyset \quad aaaa abab = \{ a,b \}$$
$$aaaa baba = \emptyset \quad aaaa abba = \{ a \}$$
$$baaa aaaa = \{ a \} \quad baaa baaa = \{ 1 \} \quad
baaa abaa = \{ a \} \quad baaa aaba = \{ a \}$$
$$baaa aaab = \{ a,b \} \quad baaa aabb = \{ a,b \} \quad
baaa bbaa = \emptyset \quad baaa abab = \{ a,b \}$$
$$baaa baba = \emptyset \quad baaa abba = \{ a \}$$
$$abaa aaaa = \{ a \} \quad abaa baaa = \emptyset \quad
abaa abaa = \{ 1,a \} \quad abaa aaba = \{ a \}$$
$$abaa aaab = \{ a,b \} \quad abaa aabb = \{ a,b \} \quad
abaa bbaa = \emptyset \quad abaa abab = \{ a,b \}$$
$$abaa baba = \emptyset \quad abaa abba = \{ a \}$$
$$aaba aaaa = \{ a \} \quad aaba baaa = \emptyset \quad
aaba abaa = \{ a \} \quad aaba aaba = \{ 1,a \}$$
$$aaba aaab = \{ a,b \} \quad aaba aabb = \{ a,b \} \quad
aaba bbaa = \emptyset \quad aaba abab = \{ a,b \}$$
$$aaba baba = \emptyset \quad aaba abba = \{ a \}$$
$$aaab aaaa = \{ a \} \quad aaab baaa = \emptyset \quad
aaab abaa = \{ a \} \quad aaab aaba = \{ a \}$$
$$aaab aaab = \{ 1,a,b \} \quad aaab aabb = \{ a,b \} \quad
aaab bbaa = \emptyset \quad aaab abab = \{ a,b \}$$
$$aaab baba = \emptyset \quad aaab abba = \{ a \}$$
$$aabb aaaa = \{ a \} \quad aabb baaa = \emptyset \quad
aabb abaa = \{ a \} \quad aabb aaba = \{ a \}$$
$$aabb aaab = \{ a,b \} \quad aabb aabb = \{ 1,a,b,t \} \quad
aabb bbaa = \emptyset \quad aabb abab = \{ a,b \}$$
$$aabb baba = \emptyset \quad aabb abba = \{ a \}$$
$$bbaa aaaa = \{ a \} \quad bbaa baaa = \emptyset \quad
bbaa abaa = \{ a \} \quad bbaa aaba = \{ a \}$$
$$bbaa aaab = \{ a,b \} \quad bbaa aabb = \{ a,b \} \quad
bbaa bbaa = \{ 1,t \} \quad bbaa abab = \{ a,b \}$$
$$bbaa baba = \emptyset \quad bbaa abba = \{ a \}$$
$$abab aaaa = \{ a \} \quad abab baaa = \emptyset \quad
abab abaa = \{ a \} \quad abab aaba = \{ a \}$$
$$abab aaab = \{ a,b \} \quad abab aabb = \{ a,b \} \quad
abab bbaa = \emptyset \quad abab abab = \{ 1,a,b,t \}$$
$$ baba = \emptyset \quad abab abba = \{ a \}$$
$$baba aaaa = \{ a \} \quad baba baaa = \emptyset \quad
baba abaa = \{ a \} \quad baba aaba = \{ a \}$$
$$baba aaab = \{ a,b \} \quad baba aabb = \{ a,b \} \quad
baba bbaa = \emptyset \quad baba abab = \{ a,b \}$$
$$baba baba = \{ 1,t \} \quad baba abba = \{ a \}$$
$$abba aaaa = \{ a \} \quad abba baaa = \emptyset \quad
abba abaa = \{ a \} \quad abba aaba = \{ a \}$$
$$abba aaab = \{ a,b \} \quad abba aabb = \{ a,b \} \quad
abba bbaa = \emptyset \quad abba abab = \{ a,b \}$$
$$abba baba = \emptyset \quad abba abba = \{ 1,a \}$$ }
By Theorem \ref{genmaintheorem} there is a bijection
between the set of elementary $H$-gradings on $K[G]$ and
$\hom(G,H)$. This result can be applied on the list
of hom-sets above. For instance the following results hold.
\begin{itemize}
\item There are no elementary $baaa$-gradings on $K[aaaa]$.

\item There is precisely one elementary $abaa$-grading on $K[aaaa]$, namely the constant grading $a$.

\item There are no nonidentity elementary $baaa$-gradings on $K[baaa]$.

\item There are precisely two elementary $aabb$-gradings on $K[baaa]$, namely the constant gradings $a$ and $b$.

\item There are no two element magmas $G$ and $H$ such that there are precisely three different elementary $H$-gradings on $K[G]$.

\item There are precisely four different elementary  $aabb$-gradings on $K[aabb]$, 
namely the identity grading 1, the constant gradings $a$ and $b$ and the grading defined by the twist map $t$.
\end{itemize}
\end{exmp}

\begin{exmp}
Let $G$ and $H$ be magmas of order two.
We use the notation from Example \ref{example3}.
By Theorem \ref{relgenmaintheorem}, there is an isomorphism
between the preordered set of elementary $H$-filters on $K[G]$
and the preordered set of submagmas of $G \times H$.
We apply this isomorphism in some particular cases.
Recall that if $f$ is a submagma of $G \times H$, then we let $F(f)$
denote the corresponding $H$-filter of $K[G]$.

Suppose that $G = H = aaaa$. Then the preordered set of
submagmas of $G \times H$ consists of the empty set
and all the eight subsets of $G \times H$ containing $(a,a)$.
These correspond to nine different elementary $H$-filters on $K[G]$, namely:
{\small $$F(\emptyset)_a = \{ 0 \} \quad F(\emptyset)_b = \{ 0 \} $$
$$F( \{ (a,a) \} )_a = Ka \quad F( \{ (a,a) \} )_b = \{ 0 \}$$
$$F( \{ (a,a),(a,b) \} )_a = Ka \quad F( \{ (a,a),(a,b) \} )_b = Ka$$
$$F( \{ (a,a),(b,a) \} )_a = Ka+Kb \quad F( \{ (a,a),(b,a) \} )_b = \{ 0 \}$$
$$F( \{ (a,a),(b,b) \} )_a = Ka \quad F( \{ (a,a),(b,b) \} )_b = Kb$$
$$F( \{ (a,a),(a,b),(b,a) \} )_a = Ka+Kb \quad F( \{ (a,a),(a,b),(b,a) \} )_b = Ka$$
$$F( \{ (a,a),(a,b),(b,b) \} )_a = Ka \quad F( \{ (a,a),(a,b),(b,b) \} )_b = Ka+Kb$$
$$F( \{ (a,a),(b,a),(b,b) \} )_a = Ka+Kb \quad F( \{ (a,a),(b,a),(b,b) \} )_b = Kb$$
$$F( \{ (a,a),(a,b),(b,a),(b,b) \} )_a = Ka+Kb \ F( \{ (a,a),(a,b),(b,a),(b,b) \} )_b = Ka+Kb$$ }
Suppose that $G = H = abba$. Then the preordered set of submagmas
of $G \times H$ consists of the empty set and all the five
subgroups of $G \times H$.
These correspond to six different elementary $H$-filters on $K[G]$, namely:
{\small $$F(\emptyset)_a = \{ 0 \} \quad F(\emptyset)_b = \{ 0 \} $$
$$F( \{ (a,a) \} )_a = Ka \quad F( \{ (a,a) \} )_b = \{ 0 \}$$
$$F( \{ (a,a),(a,b) \} )_a = Ka \quad F( \{ (a,a),(a,b) \} )_b = Ka$$
$$F( \{ (a,a),(b,a) \} )_a = Ka+Kb \quad F( \{ (a,a),(b,a) \} )_b = \{ 0 \}$$
$$F( \{ (a,a),(b,b) \} )_a = Ka \quad F( \{ (a,a),(b,b) \} )_b = Kb$$
$$F( \{ (a,a),(a,b),(b,a),(b,b) \} )_a = Ka+Kb \ F( \{ (a,a),(a,b),(b,a),(b,b) \} )_b = Ka+Kb$$ }
Suppose that $G = abaa$ and $H = aabb$. Then the there are
six submagmas of $G \times H$ which in turn correspond
to the following $H$-filters of $K[G]$:
{\small $$F(\emptyset)_a = \{ 0 \} \quad F(\emptyset)_b = \{ 0 \} $$
$$F( \{ (a,a) \} )_a = Ka \quad F( \{ (a,a) \} )_b = \{ 0 \}$$
$$F( \{ (a,b) \} )_a = \{ 0 \} \quad F( \{ (a,b) \} )_b = Ka$$
$$F( \{ (a,a),(a,b) \} )_a = Ka \quad F( \{ (a,a),(a,b) \} )_b = Ka$$
$$F( \{ (a,a),(b,a) \} )_a = Ka+Kb \quad F( \{ (a,a),(b,a) \} )_b = \{ 0 \}$$
$$F( \{ (a,a),(a,b),(b,a),(b,b) \} )_a = Ka+Kb \ F( \{ (a,a),(a,b),(b,a),(b,b) \} )_b = Ka+Kb$$ }
\end{exmp}

\begin{exmp}\label{example5}
Suppose that $n$ is a positive integer and let $G_n$
denote the zero magma having 0 and the matrices $e_{i,j}$ (see Section 1),
for $1 \leq i,j \leq n$, as elements.
Let $m$ be another positive integer.
Now we describe all nonzero elementary $G_n$-gradings of $K[G_m]$.
Take $f \in \hom_0(G_m,G_n)$.
It is easy to see that $f$
is defined by $f(0)=0$ and $f(e_{i,j}) = e_{p(i),p(j)}$,
for $1 \leq i,j \leq m$, for a unique function
$p : \{ 1,\ldots,m \} \rightarrow \{ 1,\ldots,n \}$.
By Theorem \ref{genmaintheorem} $f$ corresponds to
the nonzero $G_n$-grading of $K[G_m]$ defined by
$$K[G_m]_{e_{i,j}} = \sum_{ \begin{array}{c}
                              k \in p^{-1}(i) \\
                               l \in p^{-1}(j)
                            \end{array} } K e_{k,l}$$
for $1 \leq i,j \leq n$. In particular, there are
$n^m$ different nonzero $G_n$-gradings of $G_m$.
\end{exmp}

\begin{exmp}\label{example6}
Suppose that $n$ is a positive integer and let $G_n$
be the zero magma from Example \ref{example5}.
Let $H$ be a zero magma with the property that
its nonzero elements form a group with $q$ elements.
Now we describe all nonzero elementary $H$-gradings of $K[G_n]$.
Take $f \in \hom_0(G_n,H)$. Since $G_n$ is generated,
as a zero magma, by $\{ e_{i,i+1} \mid i = 1,\ldots,n-1 \}$
it is clear that $f$ is determined by the $q-1$ elements
$h_i = f(e_{i,i+1})$, for $i = 1,\ldots,i-1$.
By Theorem \ref{genmaintheorem} $f$ corresponds to
the nonzero $H$-grading of $K[G_n]$ defined by
$K[G_n]_h = \sum_{e_{i,j} \in f^{-1}(h)} Ke_{i,j}$,
for all $h \in H$. In particular, there are $q^{n-1}$
different nonzero elementary $H$-gradings on $K[G_n]$.
\end{exmp}

\section{Category Filtered Rings}\label{categoryfilteredrings}

In this section, we use the results from Section \ref{magmafilteredrings}
to prove category versions of Theorem \ref{genmaintheorem}
and Theorem \ref{relgenmaintheorem} (see Theorem \ref{categorymainthm}
and Theorem \ref{relcategorymainthm}).
Then we apply these results to some particular cases
of category algebras (see Examples \ref{example7}-\ref{lastexample}).
In the end of this section, we prove Theorem \ref{groupoids}.

\begin{defn}
We define a \emph{precategory} exactly as a category
except that it does not necessarily have an identity element
at each object.
We define a \emph{subprecategory} of a precategory exactly
as a subcategory except that possible identity elements
at objects of the subprecategory need not necessarily
belong to the morphisms of the subprecategory.
Recall from Freyd \cite{fre90} that a \emph{prefunctor},
between (pre)categories, is defined exactly
as a functor except that it does not necessarily
respect identity elements.
\end{defn}

\begin{defn}
Now we recall some old and define some new notions
from the theory of category graded rings
(also see e.g. \cite{lu05}, \cite{lu06}, \cite{lu07}, \cite{oinlun08}, \cite{oinlun10} or \cite{oinlun2010}).
Let $R$ be a ring and $\Gamma$ a category.
A $\Gamma$-\emph{filter} on $R$ is a set of additive subgroups,
$(V_s)_{s \in \mor(\Gamma)}$ of $R$ such that for all $s ,t \in \mor(\Gamma)$, we have
$V_s V_t \subseteq V_{st}$, if $(s,t) \in \Gamma^{(2)}$, and $V_s V_t = \{
0 \}$ otherwise.
In this article, we say that such a filter is nonzero (or strong)
if $V_s \neq \{ 0 \}$ (or $V_s V_t = V_{st}$) for all $s \in G$
(or all $(s,t) \in \Gamma^{(2)}$).
The ring $R$ is called $\Gamma$-\emph{graded}
if there is a $\Gamma$-filter, $(V_s)_{s \in \mor(\Gamma)}$ on $R$
such that $R = \oplus_{s \in \mor(\Gamma)} V_s$. If $R$ is $\Gamma$-graded and
nonzero (strong) as a $\Gamma$-system, then it is called nonzero (strongly) $\Gamma$-graded.
If $\Lambda$ is another category, then we say that a $\Lambda$-filter
$(W_t)_{t \in \mor(\Lambda)}$ on $R$ is \emph{elementary} if
$W_t = \sum_{V_s \subseteq W_t} V_s$
for all $t \in \mor(\Lambda)$.
\end{defn}

\begin{thm}\label{categorymainthm}
Let $\Gamma$ and $\Lambda$ be categories and
$R$ a ring equipped with a nonzero strong $\Gamma$-grading.
(a) There is a bijection between the set of elementary
$\Lambda$-gradings on $R$ and the set of
prefunctors from $\Gamma$ to $\Lambda$.
(b) If $\Lambda$ has the property that for each $e \in \ob(\Gamma)$,
the only idempotent of $\Gamma_e$ is the identity element from $e$ to $e$,
then there is a bijection between the set of elementary
$\Lambda$-gradings on $R$ and the set of functors
from $\Gamma$ to $\Lambda$.
The last conclusion holds in particular if $\Lambda$
is a groupoid.
\end{thm}

\begin{proof}
Suppose that $R$ is equipped with the nonzero and strong $\Gamma$-grading
$(V_s)_{s \in \mor(\Gamma)}$. Let $G$ be the magma having $\mor(\Gamma) \cup \{ 0 \}$
as elements. If we put $V_0 = \{ 0 \}$, then $(V_g)_{g \in G}$ is a nonzero magma grading on $R$.

(a) Let $(W_t)_{t \in \mor(\Lambda)}$ be an elementary $\Lambda$-grading on $R$.
If we put $H = \mor(\Lambda) \cup \{ 0 \}$ and $W_0 = \{ 0 \}$,
then $(W_h)_{h \in H}$ is a nonzero elementary magma grading on $R$.
By Theorem \ref{genmaintheorem} this $H$-grading uniquely defines a zero magma
homomorphism $M( (W_h)_{h \in H} )$ from $G$ to $H$.
But a such a homomorphism restricts uniquely to a prefunctor from $\Gamma$ to $\Lambda$.

On the other hand, given a prefunctor from $\Gamma$ to $\Lambda$, then it extends
uniquely to a zero magma homomorphism from $G$ to $H$.
This magma homomorphism defines, by Theorem \ref{genmaintheorem} again, a unique nonzero elementary
magma grading $(W_h)_{h \in H}$ on $R$.

It is clear that these constructions are inverse to each other
and hence defines a bijection between the set of nonzero elementary
$\Lambda$-gradings on $R$ and the set of
prefunctors from $\Gamma$ to $\Lambda$.

(b) The first part of the statement follows immediately from (a).
The second part of the statement follows from the first part and the
fact that the prefunctor image of an identity element is an idempotent and therefore,
by the assumptions, an identity element.
The last part of the statement follows from the second part and the fact
that the only idempotents in groups are the identity elements.
\end{proof}

\begin{thm}\label{relcategorymainthm}
If $\Gamma$ and $\Lambda$ are categories and
$R$ is a ring equipped with a nonzero strong $\Gamma$-grading,
then there is an isomorphism between the preordered set of elementary $\Lambda$-filters on $R$
and the preordered set of subprecategories of $\Gamma \times \Lambda$.
\end{thm}

\begin{proof}
Suppose that $R$ is equipped with the nonzero and strong $\Gamma$-grading
$(V_s)_{s \in \mor(\Gamma)}$. Let $G$ be the magma having $\mor(\Gamma) \cup \{ 0 \}$
as elements. If we put $V_0 = \{ 0 \}$, then $(V_g)_{g \in G}$ is a nonzero magma grading on $R$.

Let $(W_t)_{t \in \mor(\Lambda)}$ be a nonzero elementary $\Lambda$-filter on $R$.
If we put $H = \mor(\Lambda) \cup \{ 0 \}$ and $W_0 = \{ 0 \}$,
then $(W_h)_{h \in H}$ is a nonzero elementary magma filter on $R$.
By Theorem \ref{relgenmaintheorem} this $H$-filter uniquely defines a zero submagma
$M( (W_h)_{h \in H} )$ of $G \times H$.
But the nonzero elements of a zero submagma of $G \times H$ uniquely defines
a subprecategory of $\Gamma \times \Lambda$.

On the other hand, given a subprecategory of $\Gamma \times \Lambda$, then it extends
uniquely to a zero submagma of $G \times H$.
This zero submagma defines, by Theorem \ref{relgenmaintheorem} again, a unique nonzero elementary
magma filter $(W_h)_{h \in H}$ on $R$.

It is clear that these constructions are inverse to each other
and that they respect inclusion.
Therefore it defines an isomorphism between the preordered set of $\Lambda$-filters on $R$
and the preordered set of subprecategories of $\Gamma \times \Lambda$.
\end{proof}

\begin{exmp}\label{example7}
Suppose that $\Gamma$ and $\Lambda$ are finite, connected and thin cate\-gories.
If $m = |\ob(\Gamma)|$ and $n = |\ob(\Lambda)|$, then
there are $n^m$ elementary $\Lambda$-gradings on $K[\Gamma]$.

Indeed, by Theorem \ref{categorymainthm}, there is a bijection between
the set of elementary $\Lambda$-gradings on $K[\Gamma]$ and
the set of functors $\Gamma \rightarrow \Lambda$. But since
$\Gamma$ and $\Lambda$ are connected and thin, such a functor
is completely determined by its restriction to $\ob(\Gamma)$.
So we seek the cardinality of the set of functions
from $\{ 1,\ldots,m \}$ to $\{ 1,\ldots,n \}$.
Therefore there are exactly $n^m$ different functors $\Gamma \rightarrow \Lambda$.

Of these gradings, precisely $\frac{1}{n!} \sum_{i=0}^n (-1)^i
{n \choose i} (n-i)^m$ are nonzero.
In fact, this follows from the above and the
observation that an elementary $\Lambda$-grading on $K[\Gamma]$
is nonzero precisely when the corresponding functor $\Gamma \rightarrow \Lambda$
is surjective on objects. So we seek the cardinality of the set of
surjective functions from $\{ 1,\ldots,m \}$ to $\{ 1,\ldots,n \}$.
But this is the so called Stirling number which can be
calculated in accordance with the claim (see e.g. \cite{sharp}).
\end{exmp}

\begin{rem}
The first part of Example \ref{example7}
also follows directly from Example \ref{example5}
by noting that if we adjoin a zero element
to the unique connected thin category with $n$
elements, then we get the magma $G_n$.
\end{rem}

\begin{rem}
It is easy to see that if $\Gamma$ and $\Lambda$ are finite, connected and thin cate\-gories,
then an elementary $\Lambda$-grading of $K[\Gamma]$ is nonzero if and only if it
is strong. Therefore, the second part of Example \ref{example7} also
counts the number of strong elementary $\Lambda$-gradings on $K[\Gamma]$.
\end{rem}

\begin{exmp}\label{categoryfunctor}
Let $\Gamma$ be the category having two elements $a$ and $b$ and
nonidentity morphisms $\alpha : a \rightarrow a$, $\beta : a \rightarrow b$
and $\gamma : a \rightarrow b$ subject to the following relations
$\alpha^2 = {\rm id}_a$ and $\beta \alpha = \gamma$.
Let $\Lambda$ be the category having one object $c$ and one nonidentity
morphism $\delta : c \rightarrow c$ subject to the relation $\delta^2 = {\rm id}_c$.
Note that this means that $\Lambda$ is the group with two elements.
There are four different functors $F : \Gamma \rightarrow \Lambda$.
Indeed, $F$ must map all objects of $\Lambda$ to $c$.
There are four different ways for $F$ to map the morphisms of $\Gamma$, namely:
$$F(\alpha) = F(\beta) = F(\gamma) = {\rm id}_c$$
$$F(\alpha) = \delta \quad F(\beta) = {\rm id}_c \quad F(\gamma) = \delta$$
$$F(\alpha) = {\rm id}_c \quad F(\beta) = \delta \quad F(\gamma) = \delta$$
$$F(\alpha) = \delta \quad F(\beta) = \delta \quad F(\gamma) = {\rm id}_c.$$
These four functors correspond, by Theorem \ref{categorymainthm}(b), to the following
elementary $\Lambda$-gradings of $K[\Gamma]$:
$$K[\Gamma]_{{\rm id}_c} = K[\Gamma] \quad K[\Gamma]_{\delta} = \{ 0 \}$$
$$K[\Gamma]_{{\rm id}_c} = K{\rm id}_a + K{\rm id}_b + K\beta \quad K[\Gamma]_{\delta} = K\alpha + K\gamma$$
$$K[\Gamma]_{{\rm id}_c} = K{\rm id}_a + K{\rm id}_b + K\alpha \quad K[\Gamma]_{\delta} = K\beta + K\gamma$$
$$K[\Gamma]_{{\rm id}_c} = K{\rm id}_a + K{\rm id}_b + K\gamma \quad K[\Gamma]_{\delta} = K\alpha + K\beta$$
\end{exmp}

\begin{exmp}
Let the categories $\Gamma$ and $\Lambda$ be defined as in Example \ref{categoryfunctor}.
Now we illustrate Theorem \ref{relcategorymainthm} by considering two subprecategories
of $\Gamma \times \Lambda$, both of them having $\ob(\Gamma) \times \ob(\Lambda)$ as
set of objects.

First suppose that the set of morphisms of the subprecategory is:
$$\{   ({\rm id}_a,{\rm id}_c) , \ (\beta,{\rm id}_c) , \ (\beta,\delta) , \
({\rm id}_b, \ {\rm id}_c) , \ ({\rm id}_b,\delta)  \}$$
This corresponds to the following elementary $\Lambda$-filter on $K[\Gamma]$:
$$W_{{\rm id}_c} = K{\rm id}_a + K{\rm id}_b + K\beta \quad
W_{\delta} = K{\rm id}_b + K\beta$$

Next suppose that the set of morphisms of the subprecategory is:
$$\{   ({\rm id}_a,{\rm id}_c) , \ ({\rm id}_a , \delta) , \
(\alpha,{\rm id}_c) , \ (\alpha,\delta) , \
(\beta,{\rm id}_c) , \ (\beta,\delta) , \
(\gamma,{\rm id}_c) , \ (\gamma,\delta) , \
({\rm id}_b, \ {\rm id}_c) \}$$
This corresponds to the following elementary $\Lambda$-filter on $K[\Gamma]$:
$$W_{{\rm id}_c} = K{\rm id}_a + K{\rm id}_b + K\alpha + K\beta + K\gamma \quad
W_{\delta} = K{\rm id}_a + K{\rm id}_b + K\alpha + K\beta + K\gamma$$

\end{exmp}

\begin{exmp}\label{categoryfunctor1}
Let the category $\Gamma$ be defined as in Example \ref{categoryfunctor}.
Let $\Lambda$ be the category having one object $c$ and one nonidentity
morphism $\delta : c \rightarrow c$ subject to the relation $\delta^2 = \delta$.
There are eight different prefunctors $F : \Gamma \rightarrow \Lambda$.
Indeed, $F$ must map all objects of $\Lambda$ to $c$.
There are eight different ways for $F$ to map the morphisms of $\Gamma$, namely:
$$F(\alpha) = {\rm id}_c \quad F(\beta) = {\rm id}_c \quad F({\rm id}_b) = {\rm id}_c
\quad F(\gamma) = {\rm id}_c \quad F({\rm id}_a) = {\rm id}_c$$
$$F(\alpha) = \delta \quad F(\beta) = {\rm id}_c \quad F({\rm id}_b) = {\rm id}_c
\quad F(\gamma) = \delta \quad F({\rm id}_a) = \delta$$
$$F(\alpha) = {\rm id}_c \quad F(\beta) = \delta \quad F({\rm id}_b) = {\rm id}_c
\quad F(\gamma) = \delta \quad F({\rm id}_a) = {\rm id}_c$$
$$F(\alpha) = {\rm id}_c \quad F(\beta) = {\rm id}_c \quad F({\rm id}_b) = \delta
\quad F(\gamma) = {\rm id}_c \quad F({\rm id}_a) = {\rm id}_c$$
$$F(\alpha) = \delta \quad F(\beta) = \delta \quad F({\rm id}_b) = {\rm id}_c
\quad F(\gamma) = \delta \quad F({\rm id}_a) = \delta$$
$$F(\alpha) = \delta \quad F(\beta) = {\rm id}_c \quad F({\rm id}_b) = \delta
\quad F(\gamma) = \delta \quad F({\rm id}_a) = \delta$$
$$F(\alpha) = {\rm id}_c \quad F(\beta) = \delta \quad F({\rm id}_b) = \delta
\quad F(\gamma) = \delta \quad F({\rm id}_a) = {\rm id}_c$$
$$F(\alpha) = \delta \quad F(\beta) = \delta \quad F({\rm id}_b) = \delta
\quad F(\gamma) = \delta \quad F({\rm id}_a) = \delta$$
These eight functors correspond, by Theorem \ref{categorymainthm}(a), to the following
elementary $\Lambda$-gradings of $K[\Gamma]$:
$$K[\Gamma]_{{\rm id}_c} = K[\Gamma] \quad K[\Gamma]_{\delta} = \{ 0 \}$$
$$K[\Gamma]_{{\rm id}_c} = K{\rm id}_b + K\beta \quad K[\Gamma]_{\delta} = K\alpha + K\gamma + K{\rm id}_a$$
$$K[\Gamma]_{{\rm id}_c} = K{\rm id}_a + K{\rm id}_b + K\alpha \quad K[\Gamma]_{\delta} = K\beta + K\gamma$$
$$K[\Gamma]_{{\rm id}_c} = K{\rm id}_a + K\alpha + K\beta + K\gamma \quad K[\Gamma]_{\delta} = K{\rm id}_b$$
$$K[\Gamma]_{{\rm id}_c} = K{\rm id}_b \quad K[\Gamma]_{\delta} = K\alpha + K\beta + K\gamma + K{\rm id}_a$$
$$K[\Gamma]_{{\rm id}_c} = K\beta \quad K[\Gamma]_{\delta} =  K\alpha + K\gamma + K{\rm id}_a + K{\rm id}_b$$
$$K[\Gamma]_{{\rm id}_c} = K{\rm id}_a + K\alpha \quad K[\Gamma]_{\delta} = K\beta + K\gamma + K{\rm id}_b$$
$$K[\Gamma]_{{\rm id}_c} =  \{ 0 \} \quad K[\Gamma]_{\delta} = K[\Gamma]$$
\end{exmp}

\begin{exmp}\label{lastexample}
Let the categories $\Gamma$ and $\Lambda$ be defined as in Example \ref{categoryfunctor1}.
Now we again illustrate Theorem \ref{relcategorymainthm} by considering two subprecategories
of $\Gamma \times \Lambda$, both of them having $\ob(\Gamma) \times \ob(\Lambda)$ as
set of objects.

First suppose that the set of morphisms of the subprecategory is:
$$\{   ({\rm id}_a,{\rm id}_c) , \ (\beta,{\rm id}_c) , \ (\beta,\delta) , \ ({\rm id}_b,\delta)  \}$$
This corresponds to the following elementary $\Lambda$-filter on $K[\Gamma]$:
$$W_{{\rm id}_c} = K{\rm id}_a + K\beta \quad
W_{\delta} = K{\rm id}_b + K\beta$$

Next suppose that the set of morphisms of the subprecategory is:
$$\{   ({\rm id}_a,{\rm id}_c) , \ ({\rm id}_a , \delta) , \
(\alpha,{\rm id}_c) , \ (\alpha,\delta) , \
(\beta,\delta) , \ (\gamma,\delta) , \
({\rm id}_b, \ {\rm id}_c) \}$$
This corresponds to the following elementary $\Lambda$-filter on $K[\Gamma]$:
$$W_{{\rm id}_c} = K{\rm id}_a + K{\rm id}_b + K\alpha \quad
W_{\delta} = K{\rm id}_a + K{\rm id}_b + K\alpha + K\beta + K\gamma$$
\end{exmp}

\subsection*{Proof of Theorem \ref{groupoids}.}
Suppose that $\Gamma$ and $\Lambda$ are finite connected groupoids.
By Theorem \ref{categorymainthm}(b), there is a bijection between
the set of elementary $\Lambda$-gradings on $K[\Gamma]$ and the
set of functors from $\Gamma$ to $\Lambda$.
Therefore we now set out to determine the cardinality of the latter set.

First note that $\Gamma_e \cong \Gamma_{e'}$ as groups for all $e,e' \in \ob(\Gamma)$.
In fact, since $\Gamma$ is connected, there is
$E : e \rightarrow e'$ in $\mor(\Gamma)$.
Now it is easy to see that the map
$\Gamma_e \ni x \mapsto E x E^{-1} \ni \Gamma_{e'}$
is an isomorphism of groups.
Therefore, the cardinality of $\hom(\Gamma_e,\Gamma_{e'})$
is the same for all choices of $e,e' \in \ob(\Gamma)$,
and the cardinality of $\hom(\Gamma_e,\Lambda_f)$
is the same for all choices of $e \in \ob(\Gamma)$
and $f \in \ob(\Lambda)$.

Next, fix $i \in \ob(\Gamma)$. For each $j \in \ob(\Gamma)$, with $j \neq i$,
fix a morphism $\alpha_{ij} : j \rightarrow i$. Then $\Gamma$
is generated, as a groupoid, by
$X = \Gamma_i \cup \{ \alpha_{ij} \mid j \in \ob(\Gamma), \ j \neq i \}$.
In fact, supose that $\beta : k \rightarrow l$ is a morphism in $\Gamma$.
Then $\alpha_{il} \beta \alpha_{ik}^{-1} \in \Gamma_i$ and hence
$\beta \in \alpha_{il}^{-1} G_i \alpha_{ik}$.
It is also clear that there are no nontrivial products
of elements (and their inverses) in $X$ that yields an identity
element. Therefore a functor $\Gamma \rightarrow \Lambda$
is completely determined by its action on $X$.
Suppose that $e \in \ob(\Gamma)$ and $f \in \ob(\Lambda)$
and put $m = |\ob(\Lambda)|$, $n = |\ob(\Gamma)|$,
$p = |\hom(\Gamma_e,\Lambda_f)|$ and $q = |\hom(\Gamma_e,\Gamma_e)|$.
There are $m^n$ ways for the functor to map identity elements.
For each such choice there are $p q^{n-1}$ ways to map
the elements of $X$.
Therefore, there are $(p q^{n-1})^{m^n}$ elementary $\Lambda$-gradings
on the groupoid algebra $K[\Gamma]$. \hfill $\square$

\begin{rem}\label{lastremark}
If $\Gamma$ and $\Lambda$ are arbitrary
(not necessarily connected) finite groupoids,
then $\Gamma = \biguplus_{i \in I} \Gamma_i$
and $\Lambda = \biguplus_{j \in J} \Lambda_j$
for unique connected subgroupoids $\Gamma_i$,
for $i \in I$, and $\Lambda_j$, for $j \in J$,
of $\Gamma$ and $\Lambda$ respectively.
Therefore $$\hom(\Gamma,\Lambda) =
\times_{i \in I, j \in J} \hom(\Gamma_i,\Lambda_j)$$
and hence
$$| \hom(\Gamma,\Lambda) | =
\prod_{i \in I, j \in J} | \hom(\Gamma_i,\Lambda_j) |.$$
Using Theorem \ref{groupoids} on each $| \hom(\Gamma_i,\Lambda_j) |$
we can work out the cardinality of the set of
elementary $\Lambda$-gradings on $K[\Gamma]$.
\end{rem}

\end{document}